\documentclass[a4paper,11pt]{amsart}
\usepackage{amssymb,amsmath,amsthm}
\usepackage{graphicx,color}
\usepackage[normalem]{ulem}
\newcommand{\ie}{\emph{i.e.}}

\newcommand{\cf}{\emph{cf}}

\newcommand{\Real}{\mathbb{R}}

\newcommand{\eps}{\varepsilon}
\newcommand{\Cut}{\mathcal{C}}
\newcommand{\inj}{\rho}

\newcommand{\dist}{\mathop{\mathrm{dist}}\nolimits}
\newcommand{\eqskip}{ \vspace*{2mm}\\ }

\newcommand{\ds}{\displaystyle}

\newcommand\soutD{\bgroup\markoverwith
{\textcolor{red}{\rule[.5ex]{2pt}{1pt}}}\ULon}
\newtheorem{Lemma}{Lemma}
\newtheorem{Theorem}{Theorem}
\newtheorem{Corollary}{Corollary}
\newtheorem{Proposition}{Proposition}
\newtheorem{Conjecture}{Conjecture}
\theoremstyle{definition}
\newtheorem{Remark}{Remark}

\begin{document}
%
\title[Alexandrov's conjecture and the cut locus
of a surface]{
Alexandrov's isodiametric conjecture and the cut locus
of a surface}
\author{Pedro Freitas \ and \ David Krej\v{c}i\v{r}\'{\i}k}

\address{Department of Mathematics, Faculty of Human Kinetics {\rm and}
 Group of Mathematical Physics, Universidade de Lisboa, Complexo Interdisciplinar, Av.~Prof.~Gama Pinto~2, 
 P-1649-003 Lisboa, Portugal}

\email{freitas@cii.fc.ul.pt}

\address{
Department of Theoretical Physics,
Nuclear Physics Institute,
Academy of Sciences,
250\,68 \v{R}e\v{z}, Czech Republic
}
\email{krejcirik@ujf.cas.cz}

\date{\small 27 May 2014}

\thanks{The research of the first author was partially supported by FCT's project PEst-OE/MAT/UI0208/2011.
The research of the second author was supported by 
RVO61389005 and the GACR grant No.\ P203/11/0701.
}

\begin{abstract}
We prove that Alexandrov's conjecture relating the area and diameter of a convex surface holds for
the surface of a general ellipsoid. This is a direct consequence of a more general result which estimates the
deviation from the optimal conjectured bound in terms of the length of the cut locus of a point on the
surface. We also prove that the natural extension of the conjecture to general dimension holds among closed
convex spherically symmetric Riemannian manifolds. Our results are based on a new symmetrization procedure
which we believe to be interesting in its own right.
\end{abstract}

\maketitle

\section{Introduction}
%
Let~$\Sigma$ be a closed oriented surface.
(By a \emph{surface} we always mean a connected 
smooth $2$-dimensional Riemannian manifold;
\emph{closed} means compact and without boundary.)	
Unless otherwise stated, we assume that~$\Sigma$ is \emph{convex},
\ie~the Gauss curvature~$K$ is non-negative and not identically equal to zero.
Let~$A$ and~$D$ denote the surface \emph{area} 
and the (intrinsic) \emph{diameter} of~$\Sigma$, respectively.

This paper is concerned with the following conjecture
raised by A.~D.\ Alexandrov in 1955 \cite{Alexandrov_1955}:
\begin{Conjecture}\label{Conj}
For any closed oriented convex surface~$\Sigma$,
\begin{equation}
  \frac{A}{D^2} \leq \frac{\pi}{2}\approx 1.5708
  \,.
\end{equation}
\end{Conjecture}
\noindent
It can easily be seen that the value on the right hand side
of this inequality is attained 
for the respective quotient by the doubly-covered disc, 
\ie~a degenerate surface formed by
gluing two flat discs along their boundaries.

Due to the inequality relating the area and the diameter of a planar convex domain,
namely, $A/D^2\leq \pi/4$, with equality attained only for the disk, it follows that the
conjecture holds among degenerate surfaces made of two copies of the same convex
planar domain glued at the boundary.
However, for general convex surfaces the conjecture remains open with the 
strongest result so far being, 
to the best of our knowledge, that by Calabi and Cao~\cite{Calabi-Cao_1992}
yielding
\[
 \frac{A}{D^2} \leq \frac{8}{\pi}\approx 2.5465,
\]
while other (weaker) bounds may be found in~\cite{Sakai_89,Shioya_93}. 
We note that the proof of the
above inequality by Calabi and Cao is done indirectly 
by means of eigenvalue estimates.

Even within restricted families of surfaces, 
such as the boundary of parallel\-e\-pi\-peds, the problem
does not seem to be easy, mainly due to the difficulties arising in the computation of the intrinsic diameter of
a given surface. The classes for which the conjecture has been proved are tetrahedra (Makai~\cite{Makai_1973} and
Zalgaller~\cite{Zalgaller_2007}), rectangular parallelepipeds (Nikorov and Nikorova~\cite{Nikorov-Nikorova_2008}) and
surfaces of revolution (Makuha~\cite{Makuha_1966} 
and Abreu and the first author~\cite{Abreu-Freitas_2002}).
Except for the last family of surfaces, which contains the double disk, the optimizers in the first
and second cases are not degenerate, being the regular tetrahedron and the parallelepiped with edge lengths $1,1$ and
$\sqrt{2}$, respectively, with the corresponding inequalities satisfied for surfaces within each family being
\[
 \begin{array}{ccc}
  \frac{\ds A}{\ds D^2} \leq \frac{\ds 3\sqrt{3}}{\ds 4}\approx 1.299 
  & \quad \mbox{and} \quad &
  \frac{\ds A}{\ds D^2} \leq \frac{\ds 1+2\sqrt{2}}{\ds 3}\approx 1.276.
 \end{array}
\]

We point out that the result for surfaces of revolution given in~\cite{Makuha_1966,Abreu-Freitas_2002} is actually stronger,
in that the convexity restriction is replaced by the condition that the surface is diffeomorphic to the sphere and isometric
to a closed surface in $\Real^{3}$. In fact, it is suggested in~\cite{Abreu-Freitas_2002} that it might be possible
to replace convexity by this weaker condition also in the case of general surfaces.

We note that by the 
Kuiper-Nash embedding theorem, the example in Section~3.1 in~\cite{Abreu-Freitas_2002} may be $C^1$ isometrically embedded
in~$\Real^{3}$, and we thus have that there exists a surface in~$\Real^{3}$, homeomorphic to the sphere, for which the
quotient~$A/D^2$ may be made to be arbitrarily large. An interesting question is thus whether the conjecture may hold for
more general surfaces as in the case of surfaces of revolution, or whether having non-negative curvature is an essential condition.

Here we shall go one step further and, having proven the result for spherically symmetric $d$-manifolds in $\Real^{d+1}$  -- see
Theorem~\ref{Thm.Rot} in Appendix~\ref{Sec.Rot} --, we conjecture that, under some assumptions, a similar result will hold in any dimension.

In this paper we relate the quotient between $A$ and $D^2$ to the cut-locus geometry
of the surface~$\Sigma$, bounding it by the conjectured optimal value plus a deviation term proportional to the
$1$-dimensional Hausdorff measure of the cut-locus. This allows us to prove the conjecture in the case of closed
convex surfaces having one point for which the cut locus reduces to a single point -- see Corollary~\ref{Cor.single} below.
Our proof is based on a symmetrization procedure which transforms a given surface into a surface of
revolution, while ensuring that the quotient $A/D^2$ does not decrease. However, the resulting surface is not necessarily
closed, and the extra term depending on the measure of the cut locus appears as a way of estimating the length of
the boundary circle.

In order to state our results, we need some standard notions related to the geometry of a surface.
Given a point $p \in \Sigma$, let us consider
a family of geodesics~$\gamma_\theta$ emanating from~$p$
with initial direction~$\theta \in S^1$.
The \emph{distance to the cut point of~$p$ along~$\gamma_\theta$},
denoted here by $d_p(\theta)$, 
is defined as the supremum over all distances~$t$ 
for which~$\gamma_\theta$ is the minimizing geodesic,
\ie~$\dist(p,\gamma_\theta(t))=t$. 
We clearly have the global bounds 
$
  \inj \leq d_p(\theta) \leq D
$,
where~$\inj$ denotes the \emph{injectivity radius} of~$\Sigma$.
The \emph{cut locus} of~$p$,
\ie~the set of all cut points of~$p$, 
is denoted by~$\Cut_p$.
A subset of~$\Cut_p$ is formed by \emph{conjugate cut points},
\ie, roughly, those points $q \in \Cut_p$ which can 
``almost'' be joined with~$p$
via a $1$-parameter family of geodesics $\theta \mapsto \gamma_\theta$
(the precise definition is given by means of the vanishing 
of a \emph{Jacobi field} along~$\gamma_\theta$). 

It is well known that~$\Cut_p$ is a set 
of Riemannian measure zero.
Moreover, $\Cut_p$~has finite $1$-dimensional Hausdorff measure
\cite{Hebda_1994,Itoh_1996} that we denote by~$|\Cut_p|$ 
and call the \emph{total length} of~$\Cut_p$.
One of the results of this paper is the following bound.
\begin{Theorem}\label{Thm.main}
Let~$\Sigma$ be a closed oriented convex surface.
Suppose that there is a point $p \in \Sigma$ such that
the set of conjugate points in the cut locus~$\Cut_p$ is countable.
Then
\begin{equation}
  \frac{A}{D^2} \leq \frac{\pi}{2}
  + \frac{\displaystyle 
  |\Cut_p|}{\inj}
  \,.
\end{equation}
\end{Theorem}

As a consequence of this theorem, we get an estimate on the deviation from the conjectured optimal value
in terms of the cut locus of a point in $\Sigma$ with the smallest measure.
\begin{Corollary}\label{Cor.single}
Conjecture~\ref{Conj} holds true for any closed oriented convex surface~$\Sigma$ for which there is a
point~$p$ whose cut locus~$\Cut_p$ consists of a single point.
\end{Corollary}
\noindent
As a large class of surfaces to which the result applies,
let us mention rotationally symmetric surfaces;
indeed, the cut locus of a pole
reduces to its antipodal point.

More significantly, the same situation happens
for umbilical points on any ellipsoid~\cite{Itoh-Kiyohara_2004} and
we thus have
\begin{Corollary}
Conjecture~\ref{Conj} holds true for ellipsoids.
\end{Corollary}

This paper is organized as follows.
In the forthcoming Section~\ref{Sec.Pre} we summarize some
basic facts about the geometry of (not necessarily convex) surfaces which will be needed throughout the paper.
In particular, we recall the notion of geodesic polar coordinates
based on an arbitrary point of the surface
and present a formula for the total length of the cut locus of the point.
As we were unable to find a direct reference for this formula, 
we provide a proof in Appendix~\ref{Sec.App}.
The symmetrization procedure is performed in Section~\ref{Sec.proof},
where Theorem~\ref{Thm.main} is proved as a consequence
of a stronger and more general result (Theorem~\ref{Thm.strong}).
The latter does not require the hypothesis that
the conjugate cut points are countable.

In Appendix~\ref{Sec.Rot}, 
we prove a higher-dimensional analogue of Alexandrov's conjecture
for spherically symmetric manifolds in any dimension~$d \geq 2$ 
(Theorem~\ref{Thm.Rot}).
In this proof we only require that the spherically symmetric manifold 
is diffeomorphic to the sphere~$S^d$ and embedded in~$\Real^{d+1}$.
This situation includes convex spherically symmetric manifolds 
diffeomorphic to the sphere~$S^d$ due to the embeddability 
characterization of \cite[Thm.~2.1]{Rubinstein-Sinclair_2005}
combined with a classical result~\cite{Hong-Zuily_1995}
(in fact, even less restrictive, integral-type conditions
on the positivity of the curvature can be required~\cite{Engman_2004}).
Let us remark that surfaces considered in Theorems~\ref{Thm.main}
and~\ref{Thm.strong} are necessarily diffeomorphic to~$S^2$
and embeddable in~$\Real^3$. 

\section{Preliminaries}\label{Sec.Pre}
%
Let~$p$ be any point on~$\Sigma$.
A useful parameterization of~$\Sigma$
(regardless of the sign restriction of its curvature~$K$)
is given by the \emph{geodesic polar coordinates} 
\cite[Sec.~III.1]{Chavel-2nd}
based on~$p$, 
\ie~by the diffeomorphism
$$
  \Phi_p: U_p \to \Sigma\setminus\mathcal{C}_p:
  \big\{
  (t,\theta) \mapsto
  \exp_p\big( t\cos\theta \, e_1 + t \sin\theta \, e_2 \big)
  \big\}
  \,.
$$
Here~$\exp_p$ denotes the exponential map on~$T_p\Sigma$,
$\{e_1,e_2\}$ is an orthonormal basis of $T_p\Sigma$
and $U_p \subset T_p\Sigma$ is the star-shaped domain
$$
  U_p := \big\{ 
  (t\cos\theta \, e_1 + t \sin\theta \, e_2)
  \ \big| \ 0 < t < d_p(\theta) , \ \theta\in S^1
  \big\}
  \,.
$$
In these coordinates, 
the Riemannian metric of~$\Sigma$ becomes
\begin{equation}\label{metric}
  d\ell^2 = dt^2 + F_p(t,\theta)^2 \, d\theta^2 
  \,, \qquad
  (t,\theta) \in U_p
  \,.
\end{equation}
The function~$F_p^2$ admits a smooth non-negative square root~$F_p$
that plays the role of Jacobian of~$\Phi_p$
and satisfies the Jacobi equation
\begin{equation}\label{Jacobi}
\left\{
\begin{aligned}
  \partial_1^2 F_{p}(t,\theta) + K_p(t,\theta) \, F_p(t,\theta) &= 0
  && \forall (t,\theta) \in U_p
  \,, \\
  F_p(0,\theta) &= 0
  && \forall \theta \in S^1
  \,, \\
  \partial_1 F_{p}(0,\theta) &= 1
  && \forall \theta \in S^1
  \,.
\end{aligned}
\right.
\end{equation}
Here~$K_p$ is the Gauss curvature expressed 
in the geodesic polar coordinates based on~$p$.
Note that the family of geodesics~$\gamma_\theta$
mentioned in the introduction can be related 
to the exponential map via $\gamma_\theta(t)=\Phi_p(t,\theta)$.

Since the cut locus~$\Cut_p$ is a set of measure zero
with respect to the Riemannian measure of~$\Sigma$
(induced by the Euclidean Lebesgue measure on the tangent plane),
the geodesic polar coordinates represent a useful chart
for calculation of integrals.
In particular, for the surface area we have
\begin{equation}\label{area}
  A := \int_\Sigma d\Sigma =
  \int_{U_p} F_p(t,\theta) \, dt \, d\theta 
  \,,
\end{equation}
irrespectively of the choice of the point~$p$.

The diameter of~$\Sigma$, 
\begin{equation*}
  D:=\max_{p,q\in\Sigma} \dist(p,q)
  \,,
\end{equation*}
where $\dist(\cdot,\cdot)$ stands for the geodesic distance on~$\Sigma$,
is a less accessible quantity.
Using the geodesic polar coordinates, 
let us introduce
\begin{equation}\label{diameter0}
  D_p := \max_{\theta\in S^1} \, d_p(\theta)
  \,,
\end{equation}
\ie, the maximal distance of~$p$ to its cut locus measured
among all the geodesics~$\gamma_\theta$ emanating from~$p$.
Then we have
\begin{equation}\label{diameter}
  D = \max_{p \in \Sigma} \, D_p
  \,.
\end{equation}

Similarly, the injectivity radius of~$\Sigma$
can be defined through the formulae
\begin{equation}
  \inj := \min_{p \in \Sigma} \, \rho_p
  \,, \qquad \mbox{where} \qquad
  \rho_p := \min_{\theta\in S^1} \, d_p(\theta)
  \,.
\end{equation}
Note that $\inj>0$,
because the geodesic~$\gamma_\theta(t)$ is always minimizing
for sufficiently small~$t$ and~$\Sigma$ is compact.

By the Gauss-Bonnet theorem for closed surfaces,
one has the following identity 
for the \emph{total Gauss curvature}
\begin{equation}\label{GB}
  \int_\Sigma K \, d\Sigma
  = \int_{U_p} K_p(t,\theta) \, F_p(t,\theta) \, dt \, d\theta 
  = 2 \pi \, \chi_\Sigma
  \,,
\end{equation}
where~$\chi_\Sigma$ denotes the Euler characteristic of~$\Sigma$.
For orientable surfaces $\chi_\Sigma = 2 (1-g_\Sigma)$,
where~$g_\Sigma$ is the genus of~$\Sigma$.
In our case, when~$K$ is in addition non-negative and non-trivial, 
we necessarily have $g_\Sigma=0$ 
($\Sigma$~is diffeomorphic to the sphere~$S^2$)
and the total Gauss curvature thus equals~$4\pi$.

Since~$K$ is supposed to be non-negative,
it follows from~\eqref{Jacobi} that
\begin{equation}\label{comparison}
  F_p(t,\theta) \leq t 
\end{equation}
for every $(t,\theta) \in U_p$.  
In particular, $F_p$ is bounded on~$U_p$.
(For a general closed surface, regardless of the sign restriction of~$K$,
the particular bound \eqref{comparison} cannot be ensured,
but a similar bound with the right hand side being replaced
by $C t$ with a suitable constant~$C$ does hold,
and~$F_p$ is bounded on~$U_p$ in any case.)
Using in addition that~$F_p$ is uniformly continuous on~$U_p$,
we know that the function~$F_p$ admits unique boundary  
values on~$\partial U_p$ by the continuous extension.
Hence, $F_p(d_p(\theta),\theta)$ is well defined 
for every $\theta \in S^1$.
Let us also mention that $\theta \mapsto d_p(\theta)$
is a Lipschitz continuous function
\cite{Itoh-Tanaka_2000}.
As a consequence, we have the following formula for 
the total length of the cut locus
\begin{equation}\label{crucial}
  |\Cut_p| = 
  \frac{1}{2}
  \int_{S^1} 
  \sqrt{F_p\big(d_p(\theta),\theta\big)^2+d_p'(\theta)^2} 
  \, d\theta
\end{equation}
We were unable to find a direct reference for this formula, and thus
present a proof in Appendix~\ref{Sec.App}~--~\cf~Proposition~\ref{Prop.crucial}.
For this proof we need the hypothesis of Theorem~\ref{Thm.main}
about the structure of conjugate points in the cut locus.

Finally, we introduce the quantity 
\begin{equation}\label{meaning}
  M_p := 
  \int_{S^1} \frac{F_p\big(d_p(\theta),\theta\big)}{d_p(\theta)} 
  \, d\theta
  = \lim_{\eps \to 0} \frac{1}{\eps}
  \int_{S^1} 
  \int_{d_p(\theta)-\eps}^{d_p(\theta)}
  \frac{F_p(t,\theta)}{t} 
  \, dt \, d\theta
  \,.
\end{equation}
Note that the integrand is well defined because $d_p(\theta)>\inj>0$.
From~\eqref{comparison} we deduce a crude bound
\begin{equation}\label{pre.crude}
  M_p \leq 2\pi
  \,.
\end{equation}
Estimating~$d_p$ by the injectivity radius of~$\Sigma$,
we have
\begin{equation}\label{volume.crucial}
  M_p \leq \frac{1}{\rho} 
  \int_{S^1} F_p\big(d_p(\theta),\theta\big)
  \, d\theta
  = \frac{1}{\rho} \lim_{\eps \to 0} \frac{|\Omega_\eps|}{\eps}
  \,,
\end{equation}
where 
$$
  |\Omega_\eps|
  := \int_{S^1} 
  \int_{d_p(\theta)-\eps}^{d_p(\theta)}
  F_p(t,\theta)
  \, dt \, d\theta
$$ 
denotes the $2$-dimensional Riemannian measure
of an $\eps$-tubular neighbourhood of~$\Cut_p$
(note that~$\Omega_\eps$ differs from 
the $\eps$-tubular neighbourhood of~$\Cut_p$
defined by parallel curves!).
We are more interested in an estimate of the quantity~$M_p$
by means of the total length of~$\Cut_p$;
combining the inequality~\eqref{volume.crucial} with~\eqref{crucial},
we obtain
\begin{equation}\label{crude}
  M_p \leq \frac{2 \, |\Cut_p|}{\rho}
  \,.
\end{equation}
%

\section{Proof of Theorem~\ref{Thm.main}}\label{Sec.proof}
%
Our proof of Theorem~\ref{Thm.main} is based
on a symmetrization procedure. 
For every $s \in [0,D_p)$, we introduce
\begin{equation}\label{sym}
  f_p(s) :=
  \frac{1}{2\pi} \int_{S^1} \frac{D_p}{d_p(\theta)} \
  F_p\!\left(\frac{d_p(\theta)}{D_p} \, s,\theta\right) \, d\theta
  \,.
\end{equation}
Employing~\eqref{Jacobi},
it is easy to check that $f_p \in C^\infty((0,D_p))$ satisfies
\begin{equation}\label{Jacobi.bis}
\left\{
\begin{aligned}
  f_p''(s) + k_p(s) f_p(s) &= 0
  && \forall s \in (0,D_p)
  \,, \\
  f_p(0) &= 0
  \,, \\
  f_p'(0) &= 1
  \,,
\end{aligned}
\right.
\end{equation}
with
$$
  k_p(s) := \frac{\displaystyle
  \int_{S^1} \frac{d_p(\theta)}{D_p} \,
  K_p\!\left(\frac{d_p(\theta)}{D_p} s,\theta\right)
  F_p\!\left(\frac{d_p(\theta)}{D_p} s,\theta\right) \, d\theta
  }
  {\displaystyle
  \int_{S^1} \frac{D_p}{d_p(\theta)} \,
  F_p\!\left(\frac{d_p(\theta)}{D_p} s,\theta\right) 
  \, d\theta
  }
  \,.
$$
Hence, through the formula for the Riemannian metric 
on the cross-product manifold $(0,D_p) \times S^1$
$$
  d\ell^2 = ds^2 + f_p(s)^2 \, d\theta^2 
  \,, \qquad
  (s,\theta) \in (0,D_p) \times S^1
  \,,
$$
$f_p$~defines a rotationally symmetric surface~$\Sigma_p$ 
of the Gauss curvature~$k_p$.
The curvature~$k_p$ is non-negative if it is the case for~$K$.
The surface area of~$\Sigma_p$ reads
\begin{equation}\label{area.bis}
  A_p := 2\pi \int_0^{D_p} f_p(s) \, ds
  = 
  \int_{U_p}
  \frac{D_p^2}{d_p(\theta)^2} \,
  F_p(t,\theta) \, dt \, d\theta
  \,.
\end{equation}
At the same time, it is easily seen
that the symmetrization~\eqref{sym}
preserves the total Gauss curvature 
\begin{equation}\label{GB.bis}
  2\pi \int_0^{D_p} k_p(s) \, f_p(s) \, ds 
  = 
  \int_{U_p}
  K_p(t,\theta) \, F_p(t,\theta) \, dt \, d\theta
  = 4 \pi 
  \,,
\end{equation}
the last identity following from~\eqref{GB}.
On the other hand, it is absolutely necessary to stress
that (contrary to~$\Sigma$) $\Sigma_p$~may not be complete. 

Regardless of whether~$\Sigma_p$ is closed or not,
it can be represented as a surface of revolution 
embedded in~$\Real^3$. 
The embedding is explicitly provided by the mapping
$$
  \phi_p: (0,D_p) \times S^1 \to \Real^3:
  \left\{
  (s,\vartheta) \mapsto
  \big(r_p(s)\cos\theta,r_p(s)\sin\theta,z_p(s)\big)
  \right\}
  \,,
$$
where
$$
  r_p(s) := f_p(s)
  \,, \qquad
  z_p(s) := \int_0^s \sqrt{1-f_p'(\tau)^2} \, d\tau
  \,.
$$
By virtue of~\eqref{Jacobi.bis}, we have
\begin{equation}\label{fdot}
  f_p'(s) = 1 - \int_0^{s} k_p(\tau) \, f_p(\tau) \, d\tau 
  \,.
\end{equation}
Since~$k_p$ is non-negative, 
it follows with help of~\eqref{GB.bis} that 
$-1 \leq f_p' \leq 1$
(with the boundary values attained at~$D_p$ and~$0$, respectively).
Consequently, $z_p$~is well defined, non-decreasing, 
and~$\phi_p$ is thus indeed an embedding.
As mentioned above, however,
$\phi_p$ does not necessarily represent a closed surface in general.
The closedness depends on the value of $r_p(D_p)=f_p(D_p)$,
which might be different from zero.

As a matter of fact, the boundary of~$\phi_p$
is formed by a circle~$\Gamma_p$ of extrinsic radius~$f_p(D_p)$
(by extrinsic we mean that~$\Gamma_p$ is regarded
as a curve in~$\Real^3$).
Let us exclude for a moment the ``singular'' situation
$f_p(D_p)=0$ (in which case the surface~$\Sigma_p$ is closed).
By~\eqref{GB.bis} and the Gauss-Bonnet theorem 
for surfaces with boundary applied to~$\Sigma_p$, 
it then follows that 
\begin{equation}\label{geodesic}
  \oint_{\Gamma_p} \kappa_p \, dl = -2\pi
  \,,
\end{equation}
where~$\kappa_p$ denotes the geodesic curvature of~$\Gamma_p$ 
(as a curve on~$\Sigma_p \subset \Real^3$). 
Consequently, $\kappa_p = \pm f_p(D_p)^{-1}$
(with the sign depending on the parameterization of~$\Gamma_p$)
and the normal curvature of~$\Gamma_p$ is necessarily zero. 
At the same time, the length of the boundary circle equals
\begin{equation}\label{length}
  L_p := |\Gamma_p|
  = 2\pi \, f_p(D_p)
  = D_p \, M_p
  \,,
\end{equation}
where~$M_p$ was introduced in~\eqref{meaning}.
This formula remains trivially valid for $f_p(D_p)=0$.
We have the crude bound
\begin{equation}\label{length.crude}
  L_p \leq 2\pi \, D_p
  \,,
\end{equation}
which follows from~\eqref{pre.crude}.

By definition~\eqref{diameter}, $D_p \leq D$.
At the same time, comparing~\eqref{area} with~\eqref{area.bis}
and using definition~\eqref{diameter0},
we have $A_p \geq A$.
Consequently, we arrive at an intermediate bound
\begin{equation}\label{intermediate}
  \frac{A}{D^2} \leq 
  \frac{A_p}{D_p^2}
\end{equation}
valid for any $p \in \Sigma$.
The problem of Conjecture~\ref{Conj} is thus reduced 
to rotationally symmetric surfaces
(albeit possibly not complete).

To prove Theorem~\ref{Thm.main},
we divide~$A_p$ into two parts
$$
  A_p^{(1)}(\tau) := 2\pi \int_0^{\tau} f_p(s) \, ds
  \,, \qquad
  A_p^{(2)}(\tau) := 2\pi \int_{\tau}^{D_p} f_p(s) \, ds
  \,,
$$ 
with $\tau \in (0,D_p]$,
and use the following comparison arguments, respectively.
\begin{Lemma}
$
\displaystyle
  A_p^{(1)}(\tau) \leq \pi \tau^2
$.
\end{Lemma}
\begin{proof}
Since~$k_p$ is non-negative, it follows from~\eqref{Jacobi.bis}
(or more directly from~\eqref{fdot})
that $f_p(s) \leq s$ for all $s \in [0,D_p]$.
Using this estimate in the integral defining $A_p^{(1)}(\tau)$,
we get the desired inequality.
\end{proof}
\begin{Lemma}
$
\displaystyle
  A_p^{(2)}(\tau) \leq \pi (D_p-\tau)^2 + L_p (D_p-\tau)
$.
\end{Lemma}
\begin{proof}
Here the idea is to use the \emph{geodesic parallel (Fermi) coordinates}
\cite[Sec.~III.6]{Chavel-2nd} based on~$\Gamma_p$ 
rather than the geodesic polar coordinates. 
In these coordinates,
the Riemannian metric of~$\Sigma_p$ becomes
$$
  d\ell^2 = dr^2 + h_p(r)^2 \, du^2 
  \,, \qquad
  (r,u) \in (0,D_p) \times \Gamma_p
  \,,
$$
where the Jacobian $h_p \in C^\infty((0,D_p))$ satisfies 
\begin{equation}\label{Jacobi.parallel}
\left\{
\begin{aligned}
  h_p''(r) + k_p(D_p-r) h_p(r) &= 0 
  && \forall r\in(0,D_p)
  \,, \\
  h_p(0) &= 1
  \,, \\
  h_p'(0) &= -\kappa_p
  \,.
\end{aligned}
\right.
\end{equation}
Here the coordinate~$r$
is measured from~$\Gamma_p$, rather than from~$p$,
and the minus sign in front of~$\kappa_p$ is required
due to the consistency with
the convexity of~$\Sigma_p$, \cf~\eqref{geodesic}.  
Consequently, 
\begin{equation}\label{area2}
  A_p^{(2)}(\tau) = \int_{\Gamma_p} \int_{0}^{D_p-\tau} 
  h_p(r) \, dr \, du
  \,.
\end{equation}
Since~$k_p$ is non-negative, it follows from~\eqref{Jacobi.parallel}
that $h_p(r) \leq 1-\kappa_p \, r$ for all $r \in [0,D_p]$.
Using this estimate in~\eqref{area2} 
and recalling~\eqref{geodesic},
we get the desired inequality.
\end{proof}

Putting these lemmata together, we get
$$
  A_p \leq \pi \tau^2 + \pi (D_p-\tau)^2 + L_p (D_p-\tau)
$$
for any $\tau \in (0,D_p]$.
The minimum of the right hand side as a function of~$\tau$
is achieved for 
$$
  \tau_\mathrm{min} := \frac{D_p}{2} + \frac{L_p}{4\pi}
  \,, 
$$
which is an admissible point from $(0,D_p]$
due to~\eqref{length.crude}.
Using this value and expressing~$L_p$ 
in terms of~$M_p$, \cf~\eqref{length},
we get the bound
\begin{equation}
  \frac{A_p}{D_p^2}
  \leq \frac{\pi}{2} + \frac{1}{2} \, M_p 
  \left(1- \frac{M_p}{4\pi}\right)
  .
\end{equation}
Plugging this estimate into~\eqref{intermediate},
we obtain the main result of this paper.
\begin{Theorem}\label{Thm.strong}
Let~$\Sigma$ be a closed convex surface.
For any point $p\in\Sigma$, we have
\begin{equation}
  \frac{A}{D^2} \leq \frac{\pi}{2}
  + \frac{1}{2} 
  \,
  M_p 
  \left(1- \frac{M_p}{4\pi}\right)
  \,.
\end{equation}
\end{Theorem}
Theorem~\ref{Thm.main} follows as a weaker version of this theorem 
by neglecting the non-positive term in the brackets
and applying the crude bound~\eqref{crude}. 

\appendix
\section{The total length of the cut locus}\label{Sec.App}
%
This appendix is devoted to a proof of formula~\eqref{crucial}
for a general (\ie~not necessarily convex) surface.
\begin{Proposition}\label{Prop.crucial}
Let~$\Sigma$ be closed surface, 
and let~$p$ be a point in~$\Sigma$.
Assume that the set of conjugate points 
in the cut locus~$\Cut_p$ is countable.
Then the total length of the cut locus 
is given by formula~\eqref{crucial}. 
\end{Proposition}

First we recall some definitions and facts from~\cite{Hebda_1994},
to where we refer for more information.
Fix $p \in \Sigma$.
A cut point $q \in \Cut_p$ is said to be a \emph{conjugate cut point}
if it is conjugate to~$p$ along at least one minimizing
geodesic~$\gamma_\theta$ joining~$p$ to~$q$
(\ie~$F_p(t,\theta)=0$ for $t\not=0$, $\theta \in S^1$,
corresponding to~$q$),
and is said to be a \emph{non-conjugate cut point} otherwise.
The \emph{order} of a non-conjugate cut point $q \in \Cut_p$
is the number of minimizing geodesics joining~$p$ to~$q$.
The order is always finite and at least two.
A cut point is said to be a \emph{cleave point} 
if it is a non-conjugate cut of order two;
it is a \emph{non-cleave point} otherwise.

The cut locus~$\Cut_p$ is a closed, compact, 
connected and non-empty subset of~$\Sigma$.
The set of cleave points is a relatively open subset of~$\Cut_p$
forming a smooth $1$-dimensional submanifold of~$\Sigma$.
The set of non-cleave points in~$\Cut_p$ is a closed subset
whose $1$-dimensional Hausdorff measure is zero
(in fact, the Hausdorff dimension is~$0$).
It is also known that at a cleave point~$q$,
the two minimizing geodesics joining~$p$ to~$q$
make the same angle with the tangent plane at~$q$
to the submanifold of cleave points, but from opposite sides.

Moreover, by a theorem of Myers \cite[Thm.~2.3]{Hebda_1994},
the topological structure of~$\Cut_p$ is that of a local tree.
To recall the notion, two other topological definitions are needed.
An \emph{arc} is a topological space homeomorphic 
to the unit interval~$[0,1]$.
A \emph{tree} is is a topological space with the property
that every pair of points $q_1,q_2$
is contained in a unique arc with endpoints $q_1,q_2$.
A \emph{local tree} is a topological space
in which every point is contained in arbitrarily small
closed neighbourhoods which are themselves trees.
For simply connected surfaces, $\Cut_p$~is a tree.
(Convex oriented surfaces considered in this paper
are simply connected due to the Gauss-Bonnet theorem~\eqref{GB}
and the text below it; for general convex surfaces 
the tree property is also mentioned in~\cite{Zamfirescu_1998}.)

Each point of a tree is either an \emph{endpoint},
\emph{ordinary point} or \emph{branch point} 
depending upon the number (one, two, or more respectively)
of connected components possessed by a deleted neighbourhood of the point.
Every endpoint is a conjugate cut point;
every cleave point is ordinary;
and every non-conjugate, non-cleave point is a branch point.
There are at most countably many branch points in the cut locus~$\Cut_p$,
but the set of endpoints may be uncountable.
For real-analytic manifolds the endpoints are at most finite
\cite{Poincare_1905},
but there exist examples of smooth surfaces 
with an infinite number of endpoints (non-triangulable cut loci)
\cite{Gluck-Singer_1978}.
We refer to~\cite{Sinclair-Tanaka_2006} for an interesting numerical study 
of the structure of cut loci of rotationally symmetric surfaces.

\begin{proof}[Proof of Proposition~\ref{Prop.crucial}]
The map $w:S^1 \to \Cut_p$ 
defined by $w(\theta) := \Phi_p(d_p(\theta),\theta)$
is a continuous proper map (\ie~the inverse image of compact sets is compact).
Let 
\begin{align*}
  E &:= 
  \{
  \theta \in S^1 \,|\, w(\theta) 
  \mbox{ is a non-cleave point of } \Cut_p
  \}
  \,,
  \\
  E_0 &:= 
  \{
  \theta \in E \,|\, w(\theta) 
  \mbox{ is a conjugate cut point of } \Cut_p
  \}
  \,,
  \\ 
  E_1 &:= 
  \{
  \theta \in E \,|\, w(\theta) 
  \mbox{ is a non-conjugate cut point of } \Cut_p
  \}
  \,.
\end{align*}
We have $\Cut_p = w(S^1)$
and we know that the $1$-dimensional Hausdorff measure 
of the image $w(E)$ is zero
and that $V := \Cut_p \setminus w(E)$ is the union
of at most countably many smooth connected $1$-dimensional manifolds~$V_n$ 
of cleave points.
Consequently, $|\Cut_p| = |V|$.
\smallskip \\
\textbf{$\bullet$ Cleave points.}
Since the Hausdorff $1$-measure of a smooth curve is its arclength,
the standard calculus formula, 
employing the fact that $d_p$~is Lipschitz continuous~\cite{Itoh-Tanaka_2000}
(in fact, the absolute continuity established in 
\cite{Hebda_1994,Itoh_1996} would be sufficient),
gives
$$
  |V_n| = \int_{I_n^{(1)}} 
  \sqrt{F_p\big(d_p(\theta),\theta\big)^2+d_p'(\theta)^2} 
  \, d\theta
  \,.
$$
Here $I_n^{(1)} \subset S^1$ is one interval from the disjoint union 
$w^{-1}(V_n) = I_n^{(1)} \cup I_n^{(2)}$;
$w^{-1}(V_n)$ has exactly \emph{two} connected components 
because~$V_n$ is composed of cleave points.
Taking this multiplicity into account, 
we conclude with
$$
  |V| = \frac{1}{2}
  \int_{S^1 \setminus E} 
  \sqrt{F_p\big(d_p(\theta),\theta\big)^2+d_p'(\theta)^2} 
  \, d\theta
  \,.
$$
\smallskip \\
\textbf{$\bullet$ Non-conjugate non-cleave points.}
As recalled above, 
every non-conjugate non-cleave point is a branch point
and there are at most countably many branch points 
$\{q_n\}_{n=1}^N \subset \Cut_p$, with $N \in \{1,\dots,\infty\}$.
Moreover, since the order~$m_n$ of each $q_n$ is finite,
it follows that the preimage $w^{-1}(q_n)$ is composed
of $m_n$ distinct angles 
$\theta_{n}^{(1)}<\theta_{n}^{(2)}<\dots<\theta_{n}^{(m_n)}$.
Consequently, using in addition that the cut locus~$\Cut_p$ is compact,
the Lebesgue $1$-measure of
$
  E_1 = \cup_{n=1}^\infty w^{-1}(q_n)
$
is equal to zero and we have
$$
  \int_{E_1} 
  \sqrt{F_p\big(d_p(\theta),\theta\big)^2+d_p'(\theta)^2} 
  \, d\theta
  = 0
  \,.
$$
\smallskip \\
\textbf{$\bullet$ Conjugate cut points.}
For a general closed (even convex) surface 
the set of conjugate cut points can be uncountable, 
although it must be a totally disconnected set
with Hausdorff dimension being zero
\cite[Rem.~2.4]{Hebda_1994}.
Under our assumption, however, 
there are at most countably many conjugate cut points
$\{q_n\}_{n=1}^N \subset \Cut_p$, with $N \in \{1,\dots,\infty\}$.
The preimage~$E_0$ does not necessarily have the Lebesgue $1$-measure
equal to zero, however, by definition of conjugate cut points, we have
$$
  \forall \theta \in w^{-1}(q_n)
  \,, \qquad
  F_p\big(d_p(\theta),\theta\big) = 0 
  \quad\mbox{and}\quad
  d_p(\theta) = c_n
  \,,
$$
where~$c_n$ is a constant.
Consequently,
\begin{multline*}
  \int_{E_0} 
  \sqrt{F_p\big(d_p(\theta),\theta\big)^2+d_p'(\theta)^2} 
  \, d\theta
  \\
  = \sum_{n=1}^N
  \int_{w^{-1}(q_n)} 
  \sqrt{F_p\big(d_p(\theta),\theta\big)^2+d_p'(\theta)^2} 
  \, d\theta
  = 0
  \,.
\end{multline*}

\smallskip 
Summing up, the total length of the cut locus~$\Cut_p$
is determined by~$|V|$, which coincides with the formula~\eqref{crucial}
because the integration over~$E$ does not contribute. 
\end{proof}
%

\section{Alexandrov's conjecture in higher dimensions}\label{Sec.Rot}
%
Let~$\Sigma$ be the sphere~$S^d$ with a \emph{spherically symmetric}
Riemannian metric, that is, such that there exists a point $p \in \Sigma$
for which the metric in the geodesic spherical coordinates about $p$ reads
\begin{equation}\label{polar.any}
  d\ell^2 = ds^2 + f_p(s)^2 d\theta^2
  \,, 
  \qquad
  (s,\theta) \in U_p := (0,D_p) \times S^{d-1}
  \,.
\end{equation}
Here~$d\theta^2$ is the standard Euclidean metric 
on the unit sphere~$S^{d-1}$, 
$D_p$~is a positive number 
and $f_p \in C^\infty((0,D_p))$ is a positive function 
satisfying the Jacobi equation~\eqref{Jacobi.bis}, 
where~$k_p(s)$ should be interpreted
as the \emph{radial curvature} of $\Sigma$ at~$q$ 
for any~$q$ such that $\dist(p,q)=s$.
By the radial curvature we mean the restriction of 
the sectional curvature functions 
to all the planes containing the unit vector field~$W$ such that,
for any point $q\in\Sigma\setminus(\{p\}\cup\Cut_p$),
$W(q)$~is the unit vector tangent to the unique geodesic joining~$p$ to~$q$
(this notion is standardly defined 
for manifolds with a pole~\cite{GrWu},
but it clearly extends as given here 
to an arbitrary manifold when parameterized 
in the geodesic spherical coordinates).

Because of the spherical symmetry~\eqref{polar.any},
the geodesics starting from~$p$ have the same behaviour in all directions,
implying that if one geodesic is minimising in one direction up to time~$t$, 
the same must happen to all other geodesics emanating from~$p$. 
Thus the exponential map depends only on
the radial variable and we conclude that while it remains a diffeomorphism,
the boundary of its image for each such (positive) $t$ must be a $(d-1)$-sphere.
At a point for which the exponential map stops being a diffeomorphism,
the same must happen at all other points at the same distance from~$p$ and thus the cut
locus is either a $(d-1)$-sphere or a point. 
Since the closure of $\Sigma\setminus \Cut_p$ is $\Sigma$, 
which is a closed manifold,
we conclude that the cut locus of~$p$ reduces to the single point 
$p^* := \exp_p(\{D_p\} \times S^{d-1}) \in \Sigma$, 
\ie
\begin{equation}\label{polar.cut}
  \Cut_p=\{p^*\}
  \,.
\end{equation}
We note that the particular two-dimensional case of~\eqref{polar.cut} is
covered in~\cite[Lem.~2.1]{Sinclair-Tanaka_2007}.

Since $\Sigma$ is diffeomorphic to~$S^d$,
it follows from~\eqref{polar.any} and~\eqref{polar.cut} that
we may respectively identify~$p$ and~$p^*$ 
with the South and North Poles of~$S^{d}$. 
We have $\Cut_{p^*}=\{p\}$
and formulae analogous to~\eqref{polar.any} and~\eqref{Jacobi.bis} 
hold for~$p^*$.

We say that~$\Sigma$ is \emph{convex} if all the sectional curvatures
at all the points on the manifold~$\Sigma$ are non-negative.
In this case, it follows from~\eqref{Jacobi.bis} that
\begin{equation}\label{comparison.any}
  f_p(s) \leq s
  \qquad \mbox{and} \qquad
  f_{p^*}(s) \leq s
\end{equation}
for every $s \in (0,D_p)$.

We note that convexity is not necessary to obtain the above
estimates. In fact, both inequalities in~\eqref{comparison.any}~follow from 
$|f_p'(s)| \leq 1$ for every $s \in [0,D_p]$,
which is equivalent to the embeddability of~$\Sigma$ to~$\Real^{d+1}$, 
which in turn can be ensured by assuming mere positivity of integrals
of the radial curvature over polar caps of~$S^d$
(which automatically holds under the convexity assumption);
\cf~\cite[Thm.~2.1]{Rubinstein-Sinclair_2005} and~\cite{Engman_2004}.
Let us therefore assume only that~$\Sigma$ is a spherically
symmetric manifold (in the sense of the definition given above)
that is isometrically embedded in~$\Real^{d+1}$.
The manifold~$\Sigma$ can be obtained by rotating a curve 
around the~$x^1$ axis in~$\Real^{d+1}$
by the action of $SO(d)$.

Applying the comparison estimates~\eqref{comparison.any}
to the Riemannian volume of~$\Sigma$,
that we denote again by~$A$, we arrive at
\begin{align}\label{surface.any}
  A 
  &= |S^{d-1}| \int_0^{D_p} f_p^{d-1}(s) \, ds
  = |S^{d-1}| \int_0^{D_{p}} f_{p^*}^{d-1}(s) \, ds
  \nonumber \\
  &= |S^{d-1}| \int_0^{D_p/2} f_p^{d-1}(s) \, ds 
  + |S^{d-1}| \int_0^{D_{p}/2} f_{p^*}^{d-1}(s) \, ds
  \nonumber \\
  &\leq \frac{|S^{d-1}|}{2^{d-1} d} \, D_p^d
  = \frac{|B^{d}|}{2^{d-1}} \, D_p^d
  \,.
\end{align}
Here~$|S^{d-1}|$ denotes the $(d-1)$-dimensional volume 
of the $(d-1)$-dimensional sphere~$S^{d-1}$
and~$|B^{d}|$ denotes the $d$-dimensional volume 
of the $d$-dimensional unit ball~$B^{d}$. 

Since $\Sigma$ is spherically symmetric, 
its diameter~$D$ equals the distance between the poles~$p$ and $p^*$.
To see this, it is enough to consider that for any two points $a$ and $b$
in $\Sigma$, we have
\[
\begin{array}{ccl}
  \dist(a,b) 
  & \leq & 
  \min\left\{\dist(a,p)+\dist(p,b),\dist(a,p^{*})+\dist(p^{*},b) \right\}
  \eqskip
  & \leq & 
  \frac{1}{2} \big[ \dist(a,p)+\dist(p,b)+\dist(a,p^{*})+\dist(p^{*},b) \big]
  \eqskip
  & = & 
  \dist(p,p^{*})
  \,,
\end{array}
\]
where the equality uses the symmetry assumption
through the minimizing property of geodesics joining~$p$ with~$p^*$, 
namely, 
$$
  \dist(a,p) + \dist(a,p^{*}) 
  = \dist(p,p^{*}) =
  \dist(p,b) + \dist(p^{*},b) 
  \,.
$$

We have thus proven the following result.
\begin{Theorem}\label{Thm.Rot}
Let~$\Sigma$ be a closed spherically symmetric 
Riemannian manifold of dimension~$d$ 
diffeomorphic to~$S^d$
that is isometrically embedded in~$\Real^{d+1}$. 
Then
\begin{equation}\label{quotient.any}
  \frac{A}{D^d} \leq \frac{|B^{d}|}{2^{d-1}}
  \,.
\end{equation}
\end{Theorem}
\noindent
Again, it can be verified that the value on the right hand side
of this inequality is attained for the respective quotient
by the degenerate convex surface formed by
gluing two balls of diameter~$D$ along their boundaries.

It is thus natural to propose the following extension of Conjecture~\ref{Conj} to any dimension.
\begin{Conjecture}\label{Conj.any}
Inequality~\eqref{quotient.any} holds for any closed oriented
convex Riemannian manifold of dimension~$d$.
\end{Conjecture}
%

%
\begin{Remark}
Conjecture~\ref{Conj.any} makes sense also in $d=1$,
where its validity is trivial.
\end{Remark}

\end{document}